\documentclass{article}
\usepackage{hyperref}
\usepackage{cite}
\usepackage{amsfonts}
\usepackage{amsthm}
\usepackage{amsmath}
\usepackage{amssymb}
\usepackage{mathtools} 

\usepackage{authblk} 

\usepackage{xcolor}
\usepackage{physics} 

\usepackage{epsfig}

\usepackage{soul} 

\usepackage{subcaption}
\usepackage{graphicx}
\usepackage{wrapfig}
\usepackage{tikz}
\usepackage{pgfplots}
\usepackage{pgfplotstable}
\usepackage{tkz-fct}
\usetikzlibrary{pgfplots.polar}
\pgfplotsset{compat=newest}
\usepgfplotslibrary{fillbetween}
\usetikzlibrary{patterns}
\usepackage[top=1in, bottom=1in, left=1in, right=1in]{geometry}

\usepackage{mathrsfs}
\usepackage{float}
\hypersetup{colorlinks, linkcolor=red, 
filecolor=darkgreen, urlcolor=blue, citecolor=blue}
\theoremstyle{plain}
\newtheorem{Theorem}{Theorem}
\newtheorem{Lemma}{Lemma}

\newtheorem{Corollary}{Corollary}
\newtheorem{Remark}{Remark}
\theoremstyle{Definition}
\newtheorem{Definition}{Definition}

\numberwithin{equation}{section}

\definecolor{pomcol}{rgb}{1,0,0}
\definecolor{dbcol}{rgb}{0,0,0.8}

\makeatletter
\newcommand{\subjclass}[2][]{
	\let\@oldtitle\@title%
	\gdef\@title{\@oldtitle\footnotetext{#1 
			\emph{MSC $(2020):$} #2}}
}
\newcommand{\keywords}[1]{
	\let\@@oldtitle\@title%
	\gdef\@title{\@@oldtitle\footnotetext{\emph{Keywords:} #1}}
}
\makeatother

\begin{document}
	
\title{Some $m$-Fold Symmetric Bi-Univalent Function Classes and Their Associated Taylor-Maclaurin Coefficient Bounds}
\date{}

\author[1,2,3]{Hari Mohan Srivastava} 
\author[4]{Pishtiwan Othman Sabir} 
\author[5]{Sevtap Sümer Eker} 
\author[6]{Abbas Kareem Wanas} 
\author[7]{Pshtiwan Othman Mohammed\thanks{Corresponding author. 
Email: \texttt{pshtiwansangawi@gmail.com}}}
\author[8]{Nejmeddine Chorfi} 
\author[9,10,11]{Dumitru Baleanu\thanks{Corresponding author. Email: \texttt{dumitru.baleanu@gmail.com}}}
	
\affil[1]{{\small Department of Mathematics and Statistics, University of Victoria, Victoria, British Columbia V8W \; 3R4, Canada}}
\affil[2]{{\small Section of Mathematics, International Telematic University Uninettuno, I-00186 Rome, Italy}}
\affil[3]{{\small Center for Converging Humanities, Kyung Hee University, 26 Kyungheedae-ro, Dongdaemun-gu, Seoul 02447, Republic of Korea}}
\affil[4]{{\small Department of Mathematics, College of Science, University of Sulaimani, Sulaimani 46001, Kurdistan Region, Iraq}}
\affil[5]{{\small Department of Mathematics, Faculty of Science, Dicle University, TR-21280, Diyarbakır, Tutkey}}
\affil[6]{{\small Department of Mathematics, College of Science, University of Al-Qadisiyah, Al-Qadisiyah, Al Diwaniyah 58001, Iraq}}
\affil[7]{{\small Department of Mathematics, College of Education, University of Sulaimani, Sulaimani 46001, Kurdistan Region, Iraq}}
\affil[8]{{\small Department of Mathematics, College of Science, King Saud University, P.O. Box 2455, Riyadh 11451, Saudi Arabia}}
\affil[9]{\small Department of Computer Science and Mathematics, Lebanese American University, Beirut 11022801, Lebanon}
\affil[10]{\small Institute of Space Sciences, R76900 Magurele-Bucharest, Romania}
\affil[11]{\small Department of Medical Research, China Medical University, Taichung 40402, Taiwan}

	
\maketitle

\begin{abstract}
The Ruscheweyh derivative operator is used in this paper to introduce and investigate interesting general subclasses of the function class $\Sigma_{\mathrm{m}}$ of $m$-fold symmetric bi-univalent analytic functions. Estimates of the initial Taylor-Maclaurin coefficients $\left|a_{m+1}\right|$ and $\left|a_{2 m+1}\right|$ are obtained for functions of the subclasses introduced in this study, and the consequences of the results are discussed. The results presented would generalize and improve on some recent works by many earlier authors. In some cases, our estimates are better than the existing coefficient bounds. Furthermore, within the engineering domain, this paper delves into a series of complex issues related to analytic functions, $m$-fold symmetric univalent functions, and the utilization of the Ruscheweyh derivative operator. These problems encompass a broad spectrum of engineering applications, including the optimization of optical system designs, signal processing for antenna arrays, image compression techniques, and filter design for control systems. The paper underscores the crucial role of these mathematical concepts in addressing practical engineering dilemmas and fine-tuning the performance of various engineering systems. It emphasizes the potential for innovative solutions that can significantly enhance the reliability and effectiveness of engineering applications.

\end{abstract}


\section{Introduction}

Let $\mathcal{A}$ denote the class of the functions $f$ that are analytic in the open unit disk $\mathbb{U}=\{z \in \mathbb{C}:|z|<1\}$, normalized by the conditions $f(0)=f^{\prime}(0)-1=0$ of the Taylor-Maclaurin series expansion

\begin{equation}\label{1.1}
f(z)=z+\sum_{n=2}^{\infty} a_{k} z^{k}
\end{equation}

Assume $\mathcal{S}$ is the subclass of $\mathcal{A}$ that contains all univalent functions in $\mathbb{U}$ of the form (\ref{1.1}), and $\mathcal{P}$ is the subclass of all functions $h(z)$ of the form

\begin{equation}\label{1.2}
h(z)=1+h_{1} z+h_{2} z^{2}+h_{3} z^{3}+\cdots
\end{equation}

\noindent which is analytic in the open unit disk $\mathbb{U}$ and $\operatorname{Re}(h(z))>0, z \in \mathbb{U}$.

For a function $f \in \mathcal{A}$ defined by (\ref{1.1}), the Ruscheweyh derivative operator \cite{A1} is defined by

$$
\mathcal{R}^{\delta} f(z)=z+\sum_{k=2}^{\infty} \Omega(\delta, k) a_{k} z^{k},
$$

\noindent where $\delta \in \mathbb{N}_{0}=\{0,1,2, \ldots\}=\mathbb{N} \cup\{0\}, z \in \mathbb{U}$, and

$$
\Omega(\delta, k)=\frac{\Gamma(\delta+k)}{\Gamma(k) \Gamma(\delta+1)}.
$$

The Koebe 1/4-theorem \cite{A2} asserts that every univalent function $f \in \mathcal{S}$ has an inverse $f^{-1}$ defined by

$$
  f^{-1}(f(z))=z \quad(z \in \mathbb{U}) \text { and } f\left(f^{-1}(w)\right)=w \quad\left(|w|<r_{0}(f), r_{0}(f) \geqq \frac{1}{4}\right).
$$

The inverse function $g=f^{-1}$ has the form

\begin{equation} \label{1.3}
g(w)=f^{-1}(w)=w-a_{2} w^{2}+\left(2 a_{2}^{2}-a_{3}\right) w^{3}-\left(5 a_{2}^{3}-5 a_{2} a_{3}+a_{4}\right) w^{4}+\cdots.
\end{equation}

A function $f \in \mathcal{A}$ is said to be bi-univalent if both $f$ and $f^{-1}$ are univalent. The class of bi-univalent functions in $\mathbb{U}$ is denoted by $\Sigma$. The following are some examples of functions in the class $\Sigma$:

$$
\frac{z}{1-z}, \quad-\log (1-z) \quad \text {and} \quad \frac{1}{2} \log \left(\frac{1+z}{1-z}\right),
$$

\noindent with the corresponding inverse functions:

$$
\frac{e^{w}-1}{e^{w}}, \quad \frac{w}{1+w} \quad \text { and } \quad \frac{e^{2 w}-1}{e^{2 w}+1},
$$

\noindent respectively.

The estimates on the bounds of the Taylor-Maclaurin coefficients $|a_{n}|$ is an important concern problem in geometric function theory because it provides information about the geometric properties of these functions. Lewin \cite{A3} studied the class $\Sigma$ of bi-univalent functions and discovered that $\left|a_{2}\right|<1.51$ for the functions belonging to the class $\Sigma$. Later on, Brannan and Clunie \cite{A4} conjectured that $\left|a_{2}\right| \leqq \sqrt{2}$. Subsequently, Netanyahu \cite{A5} showed that max $\left|a_{2}\right|=4/3$ for $f \in$ $\Sigma$. There have been numerous recent works devoted to studying the bi-univalent functions class $\Sigma$ and obtaining nonsharp bounds on the Taylor-Maclaurin coefficients $\left|a_{2}\right|$ and $\left|a_{3}\right|$. In fact, in their oiponeering work, Srivastava et al. \cite{A6} revived as well as significantly improved the study of the analytic and bi-univalent function class $\Sigma$ in recent years. They also discovered bounds on $\left|a_{2}\right|$ and $\left|a_{3}\right|$, and they were followed by such authors (see, for example, \cite{A28,A23,A7,A8,A9} and references therein). The coefficient estimates on the bounds of $\left|a_{n}\right|$ $(n \in\{4,5,6, \ldots\})$ for a function $f \in \Sigma$ defined by (\ref{1.1}) remains an unsolved problem. In fact, for coefficients greater than three, there is no natural way to obtain an upper bound. There are a few articles where the Faber polynomial techniques were used to find upper bounds for higher-order coefficients (see, for example, \cite{A11,A12,A131,A132,A133}).

For each function $f \in \mathcal{S}$, the function

\begin{equation}\label{1.4}
h(z)=\left(f\left(z^{m}\right)\right)^{\frac{1}{m}}, \quad(z \in \mathbb{U}, m \in \mathbb{N})
\end{equation}

\noindent is univalent and maps the unit disk into a region with $m$-fold symmetry. A function $f$ is said to be $m$-fold symmetric (see \cite{A14}) and is denoted by $\mathcal{A}_{m}$, if it has the following normalized form:

\begin{equation}\label{1.5}
f(z)=z+\sum_{k=1}^{\infty} a_{m k+1} z^{m k+1}, \quad(z \in \mathbb{U}, m \in \mathbb{N}).
\end{equation}

Assume $\mathcal{S}_{m}$ denotes the class of $m$-fold symmetric univalent functions in $\mathbb{U}$ that are normalized by the series expansion (\ref{1.5}). In fact, the functions in class $\mathcal{S}$ are 1-fold symmetric. According to Koepf \cite{A14}, the $m$-fold symmetric function $h \in \mathcal{P}$ has the form

\begin{equation}\label{1.6}
h(z)=1+h_{m} z^{m}+h_{2 m} z^{2 m}+h_{3 m} z^{3 m}+\cdots.
\end{equation}

Analogous to the concept of $m$-fold symmetric univalent functions, Srivastava et al. \cite{A15} defined the concept of $m$-fold symmetric bi-univalent function in a direct way. Each function $f \in \Sigma$ generates an $m$-fold symmetric bi-univalent function for each $m \in \mathbb{N}$. The normalized form of $f$ is given as (\ref{1.5}) and the extension $g=f^{-1}$ is given by as follows:

\begin{multline}\label{1.7}
g(w)=w-a_{m+1} w^{m+1}+\left[(m+1) a_{m+1}^{2}-a_{2 m+1}\right] w^{2 m+1}- \\
\left[\frac{1}{2}(m+1)(3 m+2) a_{m+1}^{3}-(3 m+2) a_{m+1} a_{2 m+1}+a_{3 m+1}\right] w^{3 m+1}+\cdots.
\end{multline}

We denote the class of $m$-fold symmetric bi-univalent functions in $\mathbb{U}$ by $\Sigma_{\mathrm{m}}$. For $m=1$, the series (\ref{1.7}) coincides with the series expansion (\ref{1.3}) of the class $\Sigma$. Following are some examples of $m$-fold symmetric bi-univalent functions:

$$
\left[\frac{z^{m}}{1-z^{m}}\right]^{\frac{1}{m}}, \quad\left[-\log \left(1-z^{m}\right)\right]^{\frac{1}{m}} \quad \text { and } \quad\left[\frac{1}{2} \log \left(\frac{1+z^{m}}{1-z^{m}}\right)\right]^{\frac{1}{m}},
$$

\noindent with the corresponding inverse functions:

$$
\left(\frac{w^{m}}{1+w^{m}}\right)^{\frac{1}{m}}, \quad\left(\frac{e^{w^{m}}-1}{e^{w^{m}}}\right)^{\frac{1}{m}} \quad \text { and } \quad\left(\frac{e^{2 w^{m}}-1}{e^{2 w^{m}}+1}\right)^{\frac{1}{m}},
$$

\noindent respectively.

Recently, authors have expressed an interest in studying the $m$-fold symmetric bi-univalent functions class $\Sigma_{\mathrm{m}}$ (see, for example, \cite{A22,A16,A17,A17a,A19}) and obtaining non-sharp bounds estimates on the first two Taylor-Maclaurin coefficients $\left|a_{m+1}\right|$ and $\left|a_{2 m+1}\right|$.

For a function $f \in \mathcal{A}_m$ defined by (\ref{1.5}), one can think of the $m$-fold Ruscheweyh derivative operator $\mathcal{R}^\delta: \mathcal{A}_m \rightarrow \mathcal{A}_m$, which is analogous to the Ruscheweyh derivative $\mathcal{R}^\delta: \mathcal{A} \rightarrow \mathcal{A}$, and can define as follows:

$$
\mathcal{R}^{\delta} f(z)=z+\sum_{k=1}^{\infty} \frac{\Gamma(\delta+k+1)}{\Gamma(k+1) \Gamma(\delta+1)} a_{m k+1} z^{m k+1}, \quad\left(\delta \in \mathbb{N}_{0}, m \in \mathbb{N}, z \in \mathbb{U}\right).
$$

The optimization of optical systems presents engineers with the daunting task of achieving peak performance. Describing complex wavefronts necessitates the utilization of analytic and univalent functions. The critical challenge here lies in identifying functions that are not only univalent but also tailored to meet specific optical constraints, making optical system design a formidable endeavor. In the realm of signal processing for antenna arrays, the use of $m$-fold symmetric univalent functions is pivotal in beamforming. Engineers must navigate the complex landscape of electromagnetic waves to determine functions that can efficiently shape radiation patterns and combat interference, demanding both innovation and precision (see, for example, \cite{r26, r27}).

Within control systems engineering, univalent functions find application in filter design, where attaining the desired frequency response must harmonize with system stability and minimal phase distortion. Achieving this equilibrium remains a continual challenge for engineers. Modeling complex mechanical systems introduces yet another layer of engineering challenges. The Ruscheweyh derivative operator comes into play, providing a tool for analyzing functions that represent system dynamics. The results of this analysis can identify critical system parameters and drive system performance optimization. In the domain of robotics, the application of univalent functions extends to the control of robotic manipulators. Engineers must navigate the intricacies of ensuring precise and efficient control while adhering to constraints related to joint angles and velocities. Image compression and transmission are integral aspects of modern communication systems, requiring the development of efficient algorithms. Here, $m$-fold symmetric bi-univalent functions emerge as potential tools to create algorithms that optimize compression ratios while preserving image quality. Striking a delicate balance between these objectives remains an ongoing engineering challenge (see, for example, \cite{r28, r29}).

The object of this paper is to introduce new general subclasses of $m$-fold symmetric bi-univalent functions in $\mathbb{U}$ applying the $m$-fold Ruscheweyh derivative operator and obtain estimates on initial coefficients $\left|a_{m+1}\right|$ and $\left|a_{2 m+1}\right|$ for functions in subclasses $\mathcal{Q}_{\Sigma_{m}}(\tau, \lambda, \gamma, \delta ; \alpha)$ and $\Theta_{\Sigma_{m}}(\tau, \lambda, \gamma, \delta ; \beta)$ and improve some recent works of many authors. Moreover, the paper aims to provide valuable insights into the resolution of application challenges within the context of engineering, with the ultimate goal of advancing our understanding and capabilities in the above critical areas.

In order to derive our main results, we need to use the following lemma that will be useful in proving the basic theorems of section 2 and section 3.

\begin{Lemma}\label{lem1}\cite{A2}. If $h \in \mathcal{P}$ with $h(z)$ given by \eqref{1.2}, then $\left|h_{k}\right| \leqq 2$ for each $k \in \mathbb{N}$.
\end{Lemma}


\section{Coefficient Bounds for the Function Class $\mathcal{Q}_{\Sigma_{m}}(\delta, \lambda, \gamma, n ; \alpha)$}

In this section, we assume that

$$
\lambda \geqq 0, \; 0 \leqq \gamma \leqq 1, \; 0<\alpha \leqq 1, \; \tau \in \mathbb{C} \backslash\{0\}, \; \delta \in \mathbb{N}_{0} \; \text { and } \; m \in \mathbb{N}.
$$

For a function $h \in \mathcal{P}$ given by (\ref{1.2}). If $\mathcal{K}(z)$ be any complex-valued function such that $\mathcal{K}(z)=[h(z)]^{\alpha}$, then

$$
|\arg (\mathcal{K}(z))|=\alpha|\arg (h(z))|<\frac{\alpha \pi}{2}.
$$

\begin{Definition} A function $f \in \Sigma_{m}$ given by (\ref{1.5}) is called in the class $Q_{\Sigma_{m}}(\tau, \lambda, \gamma, \delta ; \alpha)$ if it satisfies the conditions:
\begin{multline}\label{2.1}
\left|\arg \left(1+\frac{1}{\tau}\left[(1-\lambda)(1-\gamma) \frac{\mathcal{R}^{\delta} f(z)}{z}\right.\right.\right.\left.\left.\left.+(\lambda(\gamma+1)+\gamma)\left(\mathcal{R}^{\delta} f(z)\right)^{\prime}+\right.\right.\right.\\
\left.\left.\left.\lambda \gamma\left(z\left(\mathcal{R}^{\delta} f(z)\right)^{\prime \prime}-2\right)-1\right]\right)\right|<\frac{\alpha \pi}{2},
\end{multline}

\noindent and

\begin{multline}\label{2.2}
 \left|\arg \left(1+\frac{1}{\tau}\left[(1-\lambda)(1-\gamma) \frac{\mathcal{R}^{\delta} g(w)}{z}\right.\right.\right.\left.\left.\left.+(\lambda(\gamma+1)+\gamma)\left(\mathcal{R}^{\delta} g(w)\right)^{\prime}+\right.\right.\right.\\
 \left.\left.\left.\lambda \gamma\left(w\left(\mathcal{R}^{\delta} g(w)\right)^{\prime \prime}-2\right)-1\right]\right)\right|  <\frac{\alpha \pi}{2},
\end{multline}

\noindent where $z, w \in \mathbb{U}$ and the function $g=f^{-1}$ is given by (\ref{1.7}).
\end{Definition}

\begin{Theorem}\label{th1}
Let $f \in Q_{\Sigma_{m}}(\tau, \lambda, \gamma, \delta ; \alpha)$ be given by (\ref{1.5}). Then

\begin{equation}\label{2.3}\small{
\left|a_{m+1}\right| \leqq \frac{2 \sqrt{2}|\tau| \alpha}{\sqrt{(\delta+1)\left|\tau \alpha(\delta+2)(m+1) \Phi_{1}(\lambda, \gamma, m)+2(1-\alpha)(\delta+1) \Phi_{2}(\lambda, \gamma, m)\right|}}},
\end{equation}

\noindent and

\begin{equation}\label{2.4}
\left|a_{2 m+1}\right| \leqq \frac{2|\tau| \alpha}{(\delta+1)(\delta+2) \Phi_{1}(\lambda, \gamma, m)}+\frac{2|\tau|^{2} \alpha^{2}(m+1)}{(\delta+1)^{2} \Phi_{2}(\lambda, \gamma, m)},
\end{equation}

\noindent where

$$
\Phi_{1}(\lambda, \gamma, m)=1+2(\lambda+\gamma) m+\lambda \gamma\left((2 m+1)^{2}+1\right),
$$

\noindent and

$$
\Phi_{2}(\lambda, \gamma, m)=\left(1+(\lambda+\gamma) m+\lambda \gamma\left((m+1)^{2}+1\right)\right)^{2}.
$$

\end{Theorem}

\begin{proof}
It follows from (\ref{2.1}) and (\ref{2.2}) that
\begin{multline}\label{2.5}
1+\frac{1}{\tau}\bigg[(1-\lambda)(1-\gamma)\frac{\mathcal{R}^{\delta} f(z)}{z}+(\lambda(\gamma+1)+\gamma)\left(\mathcal{R}^{\delta} f(z)\right)^{\prime}+\\
\lambda \gamma\left(z\left(\mathcal{R}^{\delta} f(z)\right)^{\prime \prime}-2\right)-1\bigg]=[p(z)]^{\alpha},
\end{multline}

\noindent and
\begin{multline}\label{2.6}
1+\frac{1}{\tau}\bigg[(1-\lambda)(1-\gamma)\frac{\mathcal{R}^{\delta} g(w)}{z}+(\lambda(\gamma+1)+\gamma)\left(\mathcal{R}^{\delta} g(w)\right)^{\prime}+\\
\lambda \gamma\left(w\left(\mathcal{R}^{\delta} g(w)\right)^{\prime \prime}-2\right)-1\bigg]=[q(w)]^{\alpha},
\end{multline}

\noindent where $p, q \in \mathcal{P}$ have the following representations

\begin{equation}\label{2.7}
p(z)=1+p_{m} z^{m}+p_{2 m} z^{2 m}+p_{3 m} z^{3 m}+\cdots,
\end{equation}

\noindent and

\begin{equation}\label{2.8}
q(w)=1+q_{m} w^{m}+q_{2 m} w^{2 m}+q_{3 m} w^{3 m}+\cdots.
\end{equation}

Clearly, we have
\begin{multline}\label{2.9}
[p(z)]^{\alpha}=1+\alpha p_{m} z^{m}+\bigg(\frac{1}{2} \alpha(\alpha-1) p_{m}^{2}+\alpha p_{2 m}\bigg) z^{2 m}+\\
\bigg(\frac{1}{6} \alpha(\alpha-1)(\alpha-2) p_{m}^{3}+\alpha(1-\alpha) p_{m} p_{2 m}+\alpha p_{3 m}\bigg) z^{3 m}+\cdots,
\end{multline}

\noindent and

\begin{multline}\label{2.10}
[q(w)]^{\alpha}=1+\alpha q_{m} w^{m}+\bigg(\frac{1}{2} \alpha(\alpha-1) q_{m}^{2}+\alpha q_{2 m}\bigg) w^{2 m}+\\
\bigg(\frac{1}{6} \alpha(\alpha-1)(\alpha-2) q_{m}^{3 m}+\alpha(1-\alpha) q_{m} q_{2 m}+\alpha q_{3 m}\bigg) w^{3 m}+\cdots.
\end{multline}

We also find that

\begin{multline}\label{2.11}
1+\frac{1}{\tau}\bigg[(1-\lambda)(1-\gamma) \frac{\mathcal{R}^{\delta} f(z)}{z}+(\lambda(\gamma+1)+\gamma)\big(\mathcal{R}^{\delta} f(z)\big)^{\prime}+\lambda \gamma\big(z\big(\mathcal{R}^{\delta} f(z)\big)^{\prime \prime}-2\big)-1\bigg] =\\
1+\frac{\big(1+m(\lambda+\gamma)+\lambda \gamma\big((m+1)^{2}+1\big)\big)(\delta+1)}{\tau} a_{m+1} z^{m}+ \\
\frac{\big(1+2 m(\lambda+\gamma)+\lambda \gamma\big((2 m+1)^{2}+1\big)\big)(\delta+1)(\delta+2)}{2 \tau} a_{2 m+1} z^{2 m}+\cdots,
\end{multline}

\noindent and

\begin{multline}\label{2.12}
1+\frac{1}{\tau}\bigg[(1-\lambda)(1 -\gamma) \frac{\mathcal{R}^{\delta} g(w)}{z}+(\lambda(\gamma+1)+\gamma)\big(\mathcal{R}^{\delta} g(w)\big)^{\prime}+\lambda \gamma\big(w\big(\mathcal{R}^{\delta} g(w)\big)^{\prime \prime}-2\big)-1\bigg]= \\
\quad 1-\frac{\big(1+m(\lambda+\gamma)+\lambda \gamma\big((m+1)^{2}+1\big)\big)(\delta+1)}{\tau} a_{m+1} w^{m}+ \\
\frac{\big(1+2 m(\lambda+\gamma)+\lambda \gamma\big((2 m+1)^{2}+1\big)\big)(\delta+1)(\delta+2)}{2 \tau}\big[(m+1) a_{m+1}^{2}-a_{2 m+1}\big] w^{2 m}+\cdots .
\end{multline}

Comparing the corresponding coefficients of (\ref{2.11}) and (\ref{2.12}), yields

\begin{equation}\label{2.13}
\frac{\big(1+m(\lambda+\gamma)+\lambda \gamma\big((m+1)^{2}+1\big)\big)(\delta+1)}{\tau} a_{m+1}=\alpha p_{m},
\end{equation}

\begin{equation}\label{2.14}
\frac{\big(1+2 m(\lambda+\gamma)+\lambda \gamma\big((2 m+1)^{2}+1\big)\big)(\delta+1)(\delta+2)}{2 \tau} a_{2 m+1}=\frac{\alpha(\alpha-1)}{2} p_{m}^{2}+\alpha p_{2 m},
\end{equation}

\begin{equation}\label{2.15}
-\frac{\big(1+m(\lambda+\gamma)+\lambda \gamma\big((m+1)^{2}+1\big)\big)(\delta+1)}{\tau} a_{m+1}=\alpha q_{m},
\end{equation}

\noindent and

\begin{multline} \label{2.16}
\frac{\left(1+2 m(\lambda+\gamma)+\lambda \gamma\left((2 m+1)^{2}+1\right)\right)(\delta+1)(\delta+2)}{2 \tau}\left[(m+1) a_{m+1}^{2}-a_{2 m+1}\right]=\\\frac{\alpha(\alpha-1)}{2} q_{m}^{2}+\alpha q_{2 m}.
\end{multline}

In view of (\ref{2.13}) and (\ref{2.15}), we find that

\begin{equation} \label{2.17}
p_{m}=-q_{m},
\end{equation}

\noindent and

\begin{equation}\label{2.18}
\frac{2\left(1+m(\lambda+\gamma)+\lambda \gamma\left((m+1)^{2}+1\right)\right)^{2}(\delta+1)^{2}}{\tau^{2}} a_{m+1}^{2}=\alpha^{2}\left(p_{m}^{2}+q_{m}^{2}\right).
\end{equation}

Adding (\ref{2.14}) to (\ref{2.16}) and substituting the value of $p_{m}^{2}+q_{m}^{2}$ form (\ref{2.18}) we obtain

\begin{multline}\label{2.19}
\frac{(\delta+1)(\delta+2)(m+1) \Phi_{1}(\lambda, \gamma, m)}{2 \tau} a_{m+1}^{2}=\\\frac{(\alpha-1)(\delta+1)^{2} \Phi_{2}(\lambda, \gamma, m)}{\tau^{2} \alpha} a_{m+1}^{2}+\alpha\left(p_{2 m}+q_{2 m}\right).
\end{multline}

Further computations on (\ref{2.19}) we get

\begin{equation} \label{2.20}
a_{m+1}^{2}=\frac{2 \tau^{2} \alpha^{2}\left(p_{2 m}+q_{2 m}\right)}{(\delta+1)\left[\tau \alpha(\delta+2)(m+1) \Phi_{1}(\lambda, \gamma, m)+2(1-\alpha)(\delta+1) \Phi_{2}(\lambda, \gamma, m)\right]}.
\end{equation}

Taking the absolute value of (\ref{2.20}) and applying Lemma \ref{lem1} for the coefficients $p_{2 m}$ and $q_{2 m}$, we deduce that

$$
\left|a_{m+1}\right| \leqq \frac{2 \sqrt{2}|\tau| \alpha}{\sqrt{(\delta+1)\left|\tau \alpha(\delta+2)(m+1) \Phi_{1}(\lambda, \gamma, m)+2(1-\alpha)(\delta+1) \Phi_{2}(\lambda, \gamma, m)\right|}}.
$$

Next, in order to determine the bound on $\left|a_{2 m+1}\right|$, by subtracting (\ref{2.16}) from (\ref{2.14}), we obtain

\begin{multline}\label{2.21}
\frac{\left(1+2 m(\lambda+\gamma)+\lambda \gamma\left((2 m+1)^{2}+1\right)\right)(\delta+1)(\delta+2)}{\tau} a_{2 m+1} -\\
\frac{\left(1+2 m(\lambda+\gamma)+\lambda \gamma\left((2 m+1)^{2}+1\right)\right)(\delta+1)(\delta+2)(m+1)}{2 \tau} a_{m+1}^{2}= \\
\frac{\alpha(\alpha-1)}{2}\left(p_{m}^{2}-q_{m}^{2}\right)+\alpha\left(p_{2 m}-q_{2 m}\right).
\end{multline}

Now, substituting the value of $a_{m+1}^{2}$ from (\ref{2.18}) into (\ref{2.21}) and using (\ref{2.17}), we conclude that

\begin{equation}\label{2.22}
a_{2 m+1}=\frac{\tau \alpha\left(p_{2 m}-q_{2 m}\right)}{2(\delta+1)(\delta+2) \Phi_{1}(\lambda, \gamma, m)}+\frac{\tau^{2} \alpha^{2}(m+1)\left(p_{m}^{2}+q_{m}^{2}\right)}{4(\delta+1)^{2} \Phi_{2}(\lambda, \gamma, m)}.
\end{equation}

Finally, taking the absolute value of (\ref{2.22}) and applying Lemma \ref{lem1} once again for the coefficients $p_{m}, p_{2 m}, q_{m}$ and $q_{2 m}$, we deduce that

$$
\left|a_{2 m+1}\right| \leqq \frac{2|\tau| \alpha}{(\delta+1)(\delta+2) \Phi_{1}(\lambda, \gamma, m)}+\frac{2|\tau|^{2} \alpha^{2}(m+1)}{(\delta+1)^{2} \Phi_{2}(\lambda, \gamma, m)}.
$$

This completes the proof.
\end{proof}

\section{Coefficient Bounds for the Function Class ${\Theta}_{\Sigma_{m}}(\tau, \lambda, \gamma, {\delta} ; {\beta})$}

In this section, we assume that

$$
\lambda \geqq 0, \; 0 \leqq \gamma \leqq 1, \; 0 \leqq \beta<1, \; \tau \in \mathbb{C} \backslash\{0\}, \; \delta \in \mathbb{N}_{0} \; \text { and } \; m \in \mathbb{N}.
$$

If $\mathcal{L}(z)$ be any complex-valued function such that $\mathcal{L}(z)=\beta+(1-\beta) h(z)$, then

$$
\operatorname{Re}\{\mathcal{L}(z)\}=\beta+(1-\beta) \operatorname{Re}\{h(z)\}>\beta.
$$

\begin{Definition} A function $f \in \Sigma_{m}$ given by (\ref{1.5}) is called in the class $\Theta_{\Sigma_{m}}(\tau, \lambda, \gamma, \delta ; \beta)$ if it satisfies the conditions:

\begin{multline}\label{3.1}
\operatorname{Re}\left\{1+\frac{1}{\tau}\left[(1-\lambda)(1-\gamma) \frac{\mathcal{R}^{\delta} f(z)}{z}+(\lambda(\gamma+1)+\gamma)\left(\mathcal{R}^{\delta} f(z)\right)^{\prime}+\right.\right.\\
\left.\left.\lambda \gamma\left(z\left(\mathcal{R}^{\delta} f(z)\right)^{\prime \prime}-2\right)-1\right]\right\}>\beta,
\end{multline}

\noindent and

\begin{multline}\label{3.2}
\operatorname{Re}\left\{1+\frac{1}{\tau}\left[(1-\lambda)(1-\gamma) \frac{\mathcal{R}^{\delta} g(w)}{z}+(\lambda(\gamma+1)+\gamma)\left(\mathcal{R}^{\delta} g(w)\right)^{\prime}+\right.\right.\\
\left.\left.\lambda \gamma\left(w\left(\mathcal{R}^{\delta} g(w)\right)^{\prime \prime}-2\right)-1\right]\right\}>\beta,
\end{multline}

\noindent where $z, w \in \mathbb{U}$ and the function $g=f^{-1}$ is given by (\ref{1.7}).
\end{Definition}

\begin{Theorem}\label{th2}

Let $f \in \Theta_{\Sigma_{m}}(\tau, \lambda, \gamma, \delta ; \beta)$ be given by (\ref{1.5}). Then

\begin{multline}\label{3.3}
\left|a_{m+1}\right| \leqq \min \left\{\frac{2|\tau|(1-\beta)}{(\delta+1)\left(1+m(\lambda+\gamma)+\lambda \gamma\left((m+1)^{2}+1\right)\right)},\right.\\\left. 2 \sqrt{\frac{2|\tau|(1-\beta)}{(\delta+1)(\delta+2)(m+1) \Phi_{1}(\lambda, \gamma, m)}}\right\},
\end{multline}

\noindent and

\begin{equation}\label{3.4}
\begin{aligned}
\left|a_{2 m+1}\right| \leqq \frac{4|\tau|(1-\beta)}{(\delta+1)(\delta+2)\left(1+2(\lambda+\gamma) m+\lambda \gamma\left((2 m+1)^{2}+1\right)\right)}.
\end{aligned}
\end{equation}

\end{Theorem}

\begin{proof}
It follows from (\ref{3.1}) and (\ref{3.2}) that

\begin{multline}\label{3.5}
1+\frac{1}{\tau}\left[(1-\lambda)(1-\gamma) \frac{\mathcal{R}^{\delta} f(z)}{z}+(\lambda(\gamma+1)+\gamma)\left(\mathcal{R}^{\delta} f(z)\right)^{\prime}+\right.\\
\left.\lambda \gamma\left(z\left(\mathcal{R}^{\delta} f(z)\right)^{\prime \prime}-2\right)-1\right]
=\beta+(1-\beta) p(z)
\end{multline}

\noindent and

\begin{multline}\label{3.6}
1+\frac{1}{\tau}\left[(1-\lambda)(1 -\gamma) \frac{\mathcal{R}^{\delta} g(w)}{z}+(\lambda(\gamma+1)+\gamma)\left(\mathcal{R}^{\delta} g(w)\right)^{\prime}+\right.\\
\left.\lambda \gamma\left(w\left(\mathcal{R}^{\delta} g(w)\right)^{\prime \prime}-2\right)-1\right]
 =\beta+(1-\beta) q(z)
\end{multline}

\noindent where $p(z)$ and $q(w)$ have the forms (\ref{2.7}) and (\ref{2.8}), respectively.

Clearly, we have

\begin{equation}\label{3.7}
\beta+(1-\beta) p(z)=1+(1-\beta) p_{m} z^{m}+(1-\beta) p_{2 m} z^{2 m}+(1-\beta) p_{3 m} z^{3 m}+\cdots
\end{equation}

\noindent and

\begin{equation}\label{3.8}
\beta+(1-\beta) q(w)=1+(1-\beta) q_{m} w^{m}+(1-\beta) q_{2 m} w^{2 m}+(1-\beta) q_{3 m} w^{3 m}+\cdots.
\end{equation}

Equating the corresponding coefficients of (\ref{3.5}) and (\ref{3.6}) yields

\begin{equation}\label{3.9}
\frac{\left(1+m(\lambda+\gamma)+\lambda \gamma\left((m+1)^{2}+1\right)\right)(\delta+1)}{\tau} a_{m+1}=(1-\beta) p_{m},
\end{equation}

\begin{equation}\label{3.10}
\frac{\left(1+2 m(\lambda+\gamma)+\lambda \gamma\left((2 m+1)^{2}+1\right)\right)(\delta+1)(\delta+2)}{2 \tau} a_{2 m+1}=(1-\beta) p_{2 m},
\end{equation}

\begin{equation}\label{3.11}
-\frac{\left(1+m(\lambda+\gamma)+\lambda \gamma\left((m+1)^{2}+1\right)\right)(\delta+1)}{\tau} a_{m+1}=(1-\beta) q_{m},
\end{equation}

\noindent and

\begin{equation} \label{3.12}
\frac{\left(1+2 m(\lambda+\gamma)+\lambda \gamma\left((2 m+1)^{2}+1\right)\right)(\delta+1)(\delta+2)}{2 \tau}\left[(m+1) a_{m+1}^{2}-a_{2 m+1}\right]=(1-\beta) q_{2 m}.
\end{equation}

In view of (\ref{3.9}) and (\ref{3.11}), we find that

\begin{equation} \label{3.13}
\begin{aligned}
p_{m}=-q_{m},
\end{aligned}
\end{equation}

\noindent and

\begin{equation} \label{3.14}
\begin{aligned}
\frac{2\left(1+m(\lambda+\gamma)+\lambda \gamma\left((m+1)^{2}+1\right)\right)^{2}(\delta+1)^{2}}{\tau^{2}} a_{m+1}^{2}=(1-\beta)^{2}\left(p_{m}^{2}+q_{m}^{2}\right).
\end{aligned}
\end{equation}

Adding (\ref{3.10}) to (\ref{3.12}), we obtain

\begin{multline} \label{3.15}
\frac{\left(1+2 m(\lambda+\gamma)+\lambda \gamma\left((2 m+1)^{2}+1\right)\right)(m+1)(\delta+1)(\delta+2)}{2 \tau} a_{m+1}^{2}=\\
(1-\beta)\left(p_{2 m}+q_{2 m}\right) .
\end{multline}

Hence, we find from (\ref{3.14}) and (\ref{3.15}) that

\begin{equation} \label{3.16}
\begin{aligned}
a_{m+1}^{2}=\frac{\tau^{2}(1-\beta)^{2}\left(p_{m}^{2}+q_{m}^{2}\right)}{2(\delta+1)^{2}\left(1+m(\lambda+\gamma)+\lambda \gamma\left((m+1)^{2}+1\right)\right)^{2}},
\end{aligned}
\end{equation}

\noindent and

\begin{equation}\label{3.17}
\begin{aligned}
a_{m+1}^{2}=\frac{2 \tau(1-\beta)\left(p_{2 m}+q_{2 m}\right)}{(\delta+1)(\delta+2)(m+1)\left(1+2(\lambda+\gamma) m+\lambda \gamma\left((2 m+1)^{2}+1\right)\right)},
\end{aligned}
\end{equation}

\noindent respectively. By taking the absolute value of (\ref{3.16}) and (\ref{3.17}) and applying Lemma \ref{lem1} for the coefficients $p_{m}, p_{2 m}, q_{m}$ and $q_{2 m}$, we deduce that

$$
\left|a_{m+1}\right| \leqq \frac{2|\tau|(1-\beta)}{(\delta+1)\left(1+(\lambda+\gamma) m+\lambda \gamma\left((m+1)^{2}+1\right)\right)},
$$

\noindent and

$$
\left|a_{m+1}\right| \leqq 2 \sqrt{\frac{2|\tau|(1-\beta)}{(\delta+1)(\delta+2)(m+1)\left(1+2(\lambda+\gamma) m+\lambda \gamma\left((2 m+1)^{2}+1\right)\right)^{2}}},
$$

\noindent respectively. To determine the bound on $\left|a_{2 m+1}\right|$, by subtracting (\ref{3.12}) from (\ref{3.10}), we get

\begin{multline}\label{3.18}
\frac{(\delta+1)(\delta+2)\left(1+2 m(\lambda+\gamma)+\lambda \gamma\left((2 m+1)^{2}+1\right)\right)}{\tau} a_{2 m+1}- \\
\frac{(\delta+1)(\delta+2)(m+1)\left(1+2 m(\lambda+\gamma)+\lambda \gamma\left((2 m+1)^{2}+1\right)\right)}{2 \tau} a_{m+1}^{2}= \\
(1-\beta)\left(p_{2 m}-q_{2 m}\right).
\end{multline}

Upon substituting the value of $a_{m+1}^{2}$ from (\ref{3.16}) and (\ref{3.17}) into (\ref{3.18}), we conclude that

\begin{multline}\label{3.19}
a_{2 m+1}=\frac{\tau^{2}(1-\beta)^{2}(m+1)\left(p_{m}^{2}+q_{m}^{2}\right)}{4(\delta+1)^{2}\left(1+m(\lambda+\gamma)+\lambda \gamma\left((m+1)^{2}+1\right)\right)^{2}}+\\
\frac{\tau(1-\beta)\left(p_{2 m}-q_{2 m}\right)}{(\delta+1)(\delta+2) \Phi_{1}(\lambda, \gamma, m)}
\end{multline}

\noindent and

\begin{equation}\label{3.20}
\begin{aligned}
a_{2 m+1}=\frac{2 \tau(1-\beta) p_{2 m}}{(\delta+1)(\delta+2)\left(1+2(\lambda+\gamma) m+\lambda \gamma\left((2 m+1)^{2}+1\right)\right)}.
\end{aligned}
\end{equation}

Now, taking the absolute value of (\ref{3.19}) and (\ref{3.20}) and applying Lemma \ref{lem1} once again for the coefficients $p_{m}, p_{2 m}, q_{m}$ and $q_{2 m}$, we deduce that

$$
\left|a_{2 m+1}\right| \leqq \frac{2|\tau|^{2}(1-\beta)^{2}(m+1)}{(\delta+1)^{2}\left(1+(\lambda+\gamma) m+\lambda \gamma\left((m+1)^{2}+1\right)\right)^{2}}+\frac{4|\tau|(1-\beta)}{(\delta+1)(\delta+2) \Phi_{1}(\lambda, \gamma, m)},
$$

\noindent and

$$
\left|a_{2 m+1}\right| \leqq \frac{4|\tau|(1-\beta)}{(\delta+1)(\delta+2)\left(1+2(\lambda+\gamma) m+\lambda \gamma\left((2 m+1)^{2}+1\right)\right)},
$$

\noindent respectively. This completes the proof.

\end{proof}

\section{Corollaries and Consequences}

This section is understanding to the demonstration of some special cases of the definitions and theorems. These results are given in the form of remarks and corollaries.

\begin{Remark} It should be noted that the classes $Q_{\Sigma_{m}}(\tau, \lambda, \gamma, \delta ; \alpha)$ and $\Theta_{\Sigma_{m}}(\tau, \lambda, \gamma, \delta ; \beta)$ are generalizations of well-known classes consider earlier. These classes are:
\begin{enumerate}
\item For $\delta=\gamma=0$ and $\tau=\lambda=1$, the classes $\mathcal{Q}_{\Sigma_{m}}(\tau, \lambda, \gamma, \delta ; \alpha)$ and $\Theta_{\Sigma_{m}}(\tau, \lambda, \gamma, \delta ; \beta)$ reduces to the classes $\mathcal{H}_{\Sigma, m}^{\alpha}$ and $\mathcal{H}_{\Sigma, m}(\beta)$, respectively, which were given by Srivastava et al. \cite{A15}.

\item For $\delta=\gamma=0$ and $\tau=1$, the classes $\mathcal{Q}_{\Sigma_{m}}(\tau, \lambda, \gamma, \delta ; \alpha)$ and $\Theta_{\Sigma_{m}}(\tau, \lambda, \gamma, \delta ; \beta)$ reduces to the classes $\mathcal{A}_{\Sigma, m}^{\alpha, \lambda}$ and $\mathcal{A}_{\Sigma, m}^{\lambda}(\beta)$, respectively, which were investigated recently by Eker \cite{A22}.

\item For $\delta=0$, the classes $Q_{\Sigma_{m}}(\tau, \lambda, \gamma, \delta ; \alpha)$ and $\Theta_{\Sigma_{m}}(\tau, \lambda, \gamma, \delta ; \beta)$ reduces to the classes $\mathcal{W} \mathcal{S}_{\Sigma_{m}}(\lambda, \gamma, \tau ; \alpha)$ and $\mathcal{W} \mathcal{S}_{\Sigma_{m}}^{*}(\lambda, \gamma, \tau ; \beta)$, respectively, which were considered recently by Srivastava and Wanas \cite{A20}.

\item For $\delta=\gamma=0$, the classes $\mathcal{Q}_{\Sigma_{m}}(\tau, \lambda, \gamma, \delta ; \alpha)$ and $\Theta_{\Sigma_{m}}(\tau, \lambda, \gamma, \delta ; \beta)$ reduces to the classes $\mathcal{B}_{\Sigma_{m}}(\tau, \lambda ; \alpha)$ and $\mathcal{B}_{\Sigma_{m}}^{*}(\tau, \lambda ; \beta)$, respectively, which were introduced recently by Srivastava et al. \cite{A21}.

\end{enumerate}
\end{Remark}

\begin{Remark}
In Theorem \ref{th1}, if we choose
\begin{enumerate}
\item $\delta=0$, then we obtain the results which was proven by Srivastava and Wanas \cite[ Theorem 2.1]{A20}.

\item $\delta=0$ and $\gamma=0$, then we obtain the results which was given by Srivastava et al. \cite[Theorem 2.1]{A21}.

\item $\delta=0, \gamma=0$ and $\tau=1$, then we obtain the results which was obtained by Eker \cite[Theorem 1]{A22}.

\item $\delta=0, \gamma=0, \lambda=1$ and $\tau=1$, then we obtain the results which was proven by Srivastava et al. \cite[Theorem 2]{A15}.

\end{enumerate}

\end{Remark}

By taking $\delta=0$ in Theorem \ref{th2}, we conclude the following result.

\begin{Corollary} \label{cor1} Let $f \in \Theta_{\Sigma_{m}}(\tau, \lambda, \gamma ; \beta)$ be given by \eqref{1.5}. Then

$$\tiny{
\begin{aligned}
&\left|a_{m+1}\right| \leqq \min \left\{\frac{2|\tau|(1-\beta)}{1+m(\lambda+\gamma)+\lambda \gamma\left((m+1)^{2}+1\right)},  2 \sqrt{\frac{|\tau|(1-\beta)}{(m+1)\left(1+2 m(\lambda+\gamma)+\lambda \gamma\left((2 m+1)^{2}+1\right)\right)}}\right\}
\end{aligned}}
$$
and

$$
\left|a_{2 m+1}\right| \leqq \frac{2|\tau|(1-\beta)}{1+2 m(\lambda+\gamma)+\lambda \gamma\left((2 m+1)^{2}+1\right)}.
$$
\end{Corollary}

\begin{Remark}
 The bounds on $\left|a_{m+1}\right|$ and $\left|a_{2 m+1}\right|$ given in Corollary \ref{cor1} are better than those given in \cite[Theorem 3.1]{A20}. 


\end{Remark}

By taking $\gamma=0$ in Corollary \ref{cor1}, we conclude the following result.

\begin{Corollary}\label{cor2}
 Let $f \in \Theta_{\Sigma_{m}}(\tau, \lambda ; \beta)$ be given by \eqref{1.5}. Then

$$
\left|a_{m+1}\right| \leqq \min \left\{\frac{2|\tau|(1-\beta)}{1+m \lambda}, 2 \sqrt{\frac{|\tau|(1-\beta)}{(m+1)(1+2 m \lambda)}}\right\}
$$
and

$$
\left|a_{2 m+1}\right| \leqq \frac{2|\tau|(1-\beta)}{1+2 m \lambda}.
$$
\end{Corollary}

\begin{Remark}
 The bounds on $\left|a_{m+1}\right|$ and $\left|a_{2 m+1}\right|$ given in Corollary \ref{cor2} are better than those given in \cite[Theorem 3.1]{A21}. 



\end{Remark}

By setting $\gamma=0$ and $\tau=1$ in Corollary \ref{cor1}, we conclude the following result.

\begin{Corollary} \label{cor3} Let $f \in \Theta_{\Sigma_{m}}(\lambda ; \beta)$ be given by \eqref{1.5}. Then

$$
\left|a_{m+1}\right| \leqq \min \left\{\frac{2(1-\beta)}{1+m \lambda}, 2 \sqrt{\frac{(1-\beta)}{(m+1)(1+2 m \lambda)}}\right\}
$$
and

$$
\left|a_{2 m+1}\right| \leqq \frac{2(1-\beta)}{1+2 m \lambda}.
$$
\end{Corollary}

\begin{Remark} The bounds on $\left|a_{m+1}\right|$ and $\left|a_{2 m+1}\right|$ given in Corollary \ref{cor3} are better than those given in \cite[Theorem 2]{A22}. 



\end{Remark}

By setting $\gamma=0$ and $\lambda=\tau=1$ in Corollary \ref{cor1}, we conclude the following result.

\begin{Corollary} \label{cor4} Let $f \in \Theta_{\Sigma_{m}}(\beta)$ be given by \eqref{1.5}. Then

$$
\left|a_{m+1}\right| \leqq \min \left\{\frac{2(1-\beta)}{1+m}, 2 \sqrt{\frac{(1-\beta)}{(m+1)(1+2 m)}}\right\}
$$
and

$$
\left|a_{2 m+1}\right| \leqq \frac{2(1-\beta)}{1+2 m}.
$$
\end{Corollary}

\begin{Remark} The bounds on $\left|a_{m+1}\right|$ and $\left|a_{2 m+1}\right|$ given in Corollary \ref{cor4} are better than those given in \cite[Theorem 3]{A15}. 




\end{Remark}

\begin{Remark} For 1-fold symmetric bi-univalent functions, the classes $\mathcal{Q}_{\Sigma_{1}}(\tau, \lambda, \gamma, \delta ; \alpha) \equiv Q_{\Sigma}(\tau, \lambda, \gamma, \delta ; \alpha)$ and $\Theta_{\Sigma_{1}}(\tau, \lambda, \gamma, \delta ; \beta) \equiv \Theta_{\Sigma}(\tau, \lambda, \gamma, \delta ; \beta)$ are special cases of these classes illustrated below:

\begin{enumerate}
\item For $\delta=0$, the classes $\mathcal{Q}_{\Sigma}(\tau, \lambda, \gamma, \delta ; \alpha)$ and $\Theta_{\Sigma}(\tau, \lambda, \gamma, \delta ; \beta)$ reduces to the classes $\mathcal{W} \mathcal{S}_{\Sigma}(\lambda, \gamma, \tau ; \alpha)$ and $\mathcal{W} S_{\Sigma}^{*}(\lambda, \gamma, \tau ; \beta)$, respectively, which were introduced recently by Srivastava and Wanas \cite{A20}.

\item For $\delta=\gamma=0$ and $\tau=1$, the classes $Q_{\Sigma}(\tau, \lambda, \gamma, \delta ; \alpha)$ and $\Theta_{\Sigma}(\tau, \lambda, \gamma, \delta ; \beta)$ reduces to the classes $\mathcal{B}_{\Sigma}(\alpha, \lambda)$ and $\mathcal{B}_{\Sigma}(\beta, \lambda)$, respectively, which were investigated recently by Frasin and Aouf \cite{A23}.

\item For $\delta=\gamma=0$ and $\tau=\lambda=1$, the classes $Q_{\Sigma}(\tau, \lambda, \gamma, \delta ; \alpha)$ and $\Theta_{\Sigma}(\tau, \lambda, \gamma, \delta ; \beta)$ reduces to the classes $\mathcal{H}_{\Sigma}(\alpha)$ and $\mathcal{H}_{\Sigma}(\beta)$, respectively, which were given by Srivastava et al. \cite{A6}.

\end{enumerate}
\end{Remark}

For 1-fold symmetric bi-univalent functions, Theorem \ref{th1} reduce to the following corollary:

\begin{Corollary} \label{cor5} Let 
$$f \in \mathcal{Q}_{\Sigma}(\tau, \lambda, \gamma, \delta ; \alpha)\left(\lambda \geqq 0,0 \leqq \gamma \leqq 1,0<\alpha \leqq 1, \tau \in \mathbb{C} \backslash\{0\}, \delta \in \mathbb{N}_{0}\right)$$ be given by \eqref{1.1}. Then

$$
\left|a_{2}\right| \leqq \frac{2|\tau| \alpha}{\sqrt{(\delta+1)\left|\tau \alpha(\delta+2)(1+2(\lambda+\gamma+5 \lambda \gamma))+(1-\alpha)(\delta+1)(1+\lambda+\gamma+5 \lambda \gamma)^{2}\right|}},
$$

\noindent and

$$
\left|a_{3}\right| \leqq \frac{2|\tau| \alpha}{(\delta+1)(\delta+2)(1+2(\lambda+\gamma+5 \lambda \gamma))}+\frac{4|\tau|^{2} \alpha^{2}}{(\delta+1)^{2}(1+\lambda+\gamma+5 \lambda \gamma)^{2}}.
$$
\end{Corollary}

\begin{Remark}
 In Corollary \ref{cor5}, if we choose
\begin{enumerate}
\item $\delta=0$, then we obtain the results which was given by Srivastava and Wanas \cite[ Corollary 2.1]{A20}.

\item $\delta=0, \gamma=0$ and $\tau=1$, then we obtain the results which was proven by Frasin and Aouf \cite[Theorem 2.2]{A23}.

\item $\delta=0, \gamma=0, \lambda=1$ and $\tau=1$, then we obtain the results which was obtained by Srivastava et al. \cite[Theorem 1]{A6}.
\end{enumerate}

\end{Remark}

For 1-fold symmetric bi-univalent functions, Theorem \ref{th2} reduce to the following corollary:

\begin{Corollary}\label{cor6} Let 
$$f \in \Theta_{\Sigma}(\tau, \lambda, \gamma, \delta ; \beta)\left(\lambda \geqq 0, 0 \leqq \gamma \leqq 1, 0 \leqq \beta<1, \tau \in \mathbb{C} \backslash\{0\}, \delta \in \mathbb{N}_{0}\right)$$ be given by \eqref{1.1}. Then

$$
\left|a_{2}\right| \leqq \min \left\{\frac{2|\tau|(1-\beta)}{(\delta+1)(1+\lambda+\gamma+5 \lambda \gamma)}, 2 \sqrt{\frac{|\tau|(1-\beta)}{(\delta+1)(\delta+2)(1+2(\lambda+\gamma+5 \lambda \gamma))}}\right\}
$$
and

$$
\left|a_{3}\right| \leqq \frac{4|\tau|(1-\beta)}{(\delta+1)(\delta+2)(1+2(\lambda+\gamma+5 \lambda \gamma))}.
$$
\end{Corollary}

By taking $\delta=0$ in Corollary \ref{cor6}, we have the following result.

\begin{Corollary} \label{cor7} Let 
$$f \in \Theta_{\Sigma}(\tau, \lambda, \gamma ; \beta)(\lambda \geqq 0, 0 \leqq \gamma \leqq 1, 0 \leqq \beta<1, \tau \in \mathbb{C} \backslash\{0\})$$ be given by \eqref{1.1}. Then

$$
\left|a_{2}\right| \leqq \min \left\{\frac{2|\tau|(1-\beta)}{1+\lambda+\gamma+5 \lambda \gamma}, \sqrt{\frac{2|\tau|(1-\beta)}{1+2(\lambda+\gamma+5 \lambda \gamma)}}\right\}
$$
and

$$
\left|a_{3}\right| \leqq \frac{2|\tau|(1-\beta)}{1+2(\lambda+\gamma+5 \lambda \gamma)}.
$$
\end{Corollary}

\begin{Remark} The bounds on $\left|a_{2}\right|$ and $\left|a_{3}\right|$ given in Corollary \ref{cor7} are better than those given in \cite[Corollary 3.1]{A20}. 



\end{Remark}

By setting $\delta=\gamma=0$ and $\tau=1$ in Corollary \ref{cor6}, we conclude the following result.

\begin{Corollary} \label{cor8} Let 
$$f \in \Theta_{\Sigma}(\lambda; \beta)(\lambda \geqq 0, 0 \leqq \beta<1)$$ be given by \eqref{1.1}. Then

$$
\left|a_{2}\right| \leqq \min \left\{\frac{2(1-\beta)}{1+\lambda}, \sqrt{\frac{2(1-\beta)}{1+2 \lambda}}\right\}
$$
and

$$
\left|a_{3}\right| \leqq \frac{2(1-\beta)}{1+2 \lambda}.
$$
\end{Corollary}

\begin{Remark} The bounds on $\left|a_{2}\right|$ and $\left|a_{3}\right|$ given in Corollary \ref{cor8} are better than those given in \cite[Theorem 3.2]{A23}. 



\end{Remark}

By setting $\delta=\gamma=0$ and $\lambda=\tau=1$ in Corollary \ref{cor6}, we conclude the following result.

\begin{Corollary}\label{cor9} Let $f \in \Theta_{\Sigma}(\beta)(0 \leqq \beta<1)$ be given by 
\eqref{1.1}. Then

$$
\left|a_{2}\right| \leqq \min \left\{1-\beta, \sqrt{\frac{2(1-\beta)}{3}}\right\}
$$
and

$$
\left|a_{3}\right| \leqq \frac{2(1-\beta)}{3}.
$$
\end{Corollary}

\begin{Remark} The bounds on $\left|a_{2}\right|$ and $\left|a_{3}\right|$ given in Corollary \ref{cor9} are better than those given in \cite[Theorem 2]{A6}. 



\end{Remark}



\section*{Data Availability}
Data sharing not applicable to this article as no datasets were generated or analyzed during the current study.

\section*{Funding}
Not applicable.

\section*{Author's contributions} 
Conceptualization, H.M.S., D.B. and P.O.S.; Formal analysis, S.S.E.; Funding acquisition, A.K.W.; Investigation, H.M.S., P.O.S., D.B. and P.O.M.; Methodology, P.O.S, S.S.E. and D.B.; Project administration, A.K.W. and N.C.; Resources, P.O.S. and S.S.E.; Software, P.O.M.; Supervision, P.O.S. and S.S.E.; Validation, A.K.W. and N.C.; Visualization, N.C. and D.B.; Writing – original draft, H.M.S. and P.O.M.; Writing – review \& editing, P.O.S. and N.C. All of the authors read and approved the final manuscript.

\section*{Acknowledgements}
Researchers Supporting Project number (RSP2023R153), King Saud University, Riyadh, Saudi Arabia.

\section*{Declarations}

\noindent {\bf Ethical approval}\; Not applicable.

\noindent {\bf Competing interests}\; The authors declare no competing interests.


\end{document}